\documentclass[12pt,onecolumn]{article}
\usepackage{amssymb}
\usepackage{amsfonts}
\usepackage{amsmath}

\newtheorem{theorem}{Theorem}

\newtheorem{example}{Example}

\newenvironment{proof}[1][Proof]{\noindent\textbf{#1.} }{\ \rule{0.5em}{0.5em}}
\input{tcilatex}
\begin{document}

\title{General proof for irrationality of infinite sums based on Fourier's
proof}
\author{Tomer Shushi \and {\small Department of Physics, Ben-Gurion
University of the Negev, Beersheba 8410501, Israel}}
\maketitle

\begin{abstract}
In this paper we review a general proof for the irrationality property of
numbers which take a certain form of infinite sums.

\end{abstract}

\bigskip {\large Introduction:}

Proofs for irrationality are sometimes very difficult to obtain, for
instance the proofs of the irrationality of $\pi /e$ or $\pi +e$ as well as $%
\pi -e$ are open problems in mathematics. The proofs for the irrationality
of different numbers are sometimes based on different techniques and
approaches (see, for instance, [1,2,3,4]). The origins of the proof in this
paper seems to be unknown.  

\bigskip Define a number $P_{\chi }$ as the infinite sum%
\begin{equation*}
P_{\chi }=\overset{\infty }{\underset{n=0}{\sum }}\frac{\chi (n)}{n!}.
\end{equation*}%
For any $n=0,1,...,$ $\chi (n)$ is a positive integer, and $0\leq \chi
(n)\leq M\ $where $M$ is a known number (for instance $M=100$ so $0\leq \chi
(n)\leq 100\ \forall n=0,1,...$). Note that for finite $M,$ $P_{\chi }$ is a
finite number, since%
\begin{equation*}
\overset{\infty }{\underset{n=0}{\sum }}\frac{\chi (n)}{n!}\leq M\overset{%
\infty }{\underset{n=0}{\sum }}\frac{1}{n!}=M\cdot e.
\end{equation*}

\bigskip

\begin{theorem}
$P_{\chi }$\ is irrational number if the assumption that this number is
rational implies that its denominator should be greater than $M.$
\end{theorem}

\begin{proof}
To prove the Theorem we first assume that $P_{\chi }$ is rational, i.e. $%
P_{\chi }$ can be written as the division between two integers,%
\begin{equation*}
P_{\chi }=\overset{\infty }{\underset{n=0}{\sum }}\frac{\chi (n)}{n!}=\frac{a%
}{b}.
\end{equation*}%
Now we define a number $X$%
\begin{equation}
X:=b!\left( \overset{\infty }{\sum_{n=0}}\frac{\chi (n)}{n!}-\overset{b}{%
\sum_{n=0}}\frac{\chi (n)}{n!}\right) =b!\overset{\infty }{\sum_{n=b+1}}%
\frac{\chi (n)}{n!}>0,  \label{01}
\end{equation}%
and recall that we assumed that $P_{\chi }$\ is a rational number $a/b,$
therefore we argue that $X$ can also get only non-negative integer numbers
since%
\begin{eqnarray*}
X &=&b!\left( \overset{\infty }{\sum_{n=0}}\frac{\chi (n)}{n!}-\overset{b}{%
\sum_{n=0}}\frac{\chi (n)}{n!}\right) =\frac{ab!}{b}-\overset{b}{\sum_{n=0}}%
\frac{\chi (n)b!}{n!} \\
&=&a(b-1)!-\overset{b}{\sum_{n=0}}\frac{\chi (n)b!}{n!}\geq 1.
\end{eqnarray*}%
Furthermore, the following inequality is well known%
\begin{equation*}
\frac{b!}{n!}<\frac{1}{(1+b)^{n-b}},\text{ }n\geq b+1,
\end{equation*}%
and because $\chi (n)\leq M,\forall n\geq 1$ it is also true that%
\begin{equation}
\frac{b!\chi (n)}{n!}<\frac{M}{(1+b)^{n-b}},\text{ }n\geq b+1.  \label{02}
\end{equation}%
Substituting (\ref{02}) in (\ref{01}) we get the inequality%
\begin{eqnarray*}
X &=&\overset{\infty }{\sum_{n=b+1}}\frac{b!\chi (n)}{n!}<M\overset{\infty }{%
\sum_{n=b+1}}\frac{1}{(1+b)^{n-b}}=M\overset{\infty }{\sum_{n=1}}\frac{1}{%
(1+b)^{n}} \\
&=&M\left( \frac{1}{1-1/\left( 1+b\right) }\right) \frac{1}{b+1}=\frac{M}{b},
\end{eqnarray*}%
therefore for $b>M$ it is clear that%
\begin{equation*}
X<1,
\end{equation*}%
and we get a contradiction since $X\geq 1.$ Therefore $P_{\chi }$ is an
irrational number.
\end{proof}

\begin{example}
Suppose that $P_{\chi }$ takes the form 
\begin{equation*}
P_{\chi }=\frac{3}{0!}+\frac{5}{1!}+\frac{7}{2!}+\frac{3}{3!}+\frac{5}{4!}%
+...\text{ },
\end{equation*}%
where $\chi (n)$ is periodic with the values $3,5,$ and $7$. Since in this
case $b$ should be greater than $7$ (we can simply check it!) then $P_{\chi
} $ is an irrational number.
\end{example}

\begin{example}
\bigskip Any number $P_{\chi }-e$ where $0<\chi (n)\leq M\ \forall n\geq 1,$ 
$P_{\chi }-e>0$ is an irrational number if $P_{\chi }$\ is an irrational
number, proved by the previous Theorem.
\end{example}

\begin{example}
For a periodic sequence $P_{k}$ of prime numbers from $p_{1}=2$ to $%
p_{d}=101,$ we can write the following number%
\begin{equation*}
P_{\chi }=\frac{2}{0!}+\frac{3}{1!}+\frac{5}{2!}+\frac{7}{3!}+\frac{11}{4!}%
+...+\frac{101}{5!}+\frac{97}{6!}+\frac{89}{7!}+...\text{ }.
\end{equation*}%
This periodic sum is an irrational number since it follows the Theorem.
\end{example}

\begin{example}
The following number%
\begin{equation*}
\overset{\infty }{\underset{n=0}{\sum }}\frac{\cos (n\pi )}{n!}+e
\end{equation*}%
is also irrational since $\cos (j\pi ),$ $j=1,2,...,$ can get the values $-1$
or $1$ and therefore in this case $\chi (j)=0,2$ which implies on the
irrationality of this number.
\end{example}

$\ $

\end{document}